\documentclass[11pt,a4paper, oneside]{amsart}
\usepackage{amssymb}
\usepackage[latin2]{inputenc}
\usepackage{graphicx}
\usepackage{indentfirst}

\newtheorem{Def}{Definition}
\newtheorem{Lem}{Lemma}

\newtheorem{Thm}{Theorem}
\newtheorem{Cor}{Corollary}
\newtheorem{Rem}{Remark}

\title[Analytic properties and the asymptotic behavior...]{Analytic properties and the asymptotic behavior of the area function of a Funk metric}
\author{Cs. Vincze}
\address{Inst. of Math., Univ. of Debrecen \\
H-4010 Debrecen, P.O.Box 12 \\
Hungary}
\email{csvincze@science.unideb.hu}

\begin{document}
\begin{abstract}
In Minkowski geometry the unit ball is a compact convex body $K$ containing the origin in its interior. The boundary of the body is formed by the unit vectors. We also have a so-called Minkowski functional to measure the length of vectors. By changing the origin in the interior of the body we have a smoothly varying family of Minkowski functionals. This is called the Funk metric. Under some regularity conditions the Minkowski functionals allow us to measure the volume (area) of the indicatrix bodies (hypersurfaces). Some homogenity properties provide the volume and the area to be proportional. The area as the function of the base point varying in the interior of $K$ is strictly convex \cite{Vincze3}. This is called the area function of the Funk manifold. If the minimum is attained at the origin then $K$ is said to be balanced. The idea comes from the generalization of Brickell's theorem \cite{Brickell} for Finsler manifolds with balanced indicatrices \cite{Vincze3}. As a continuation of \cite{Vincze3} we are going to investigate analytic properties and the asymptotic behavior of the area function of a Funk manifold. We prove that the area function is locally analytic and the area can be arbitrary large near to the boundary of $K$. Therefore the minimum always attained at a uniquely determined interior point of $K$. If we apply the result to the indicatrices of a Finsler manifold point by point then the uniquely defined minima of the area functions constitute a vector field. We prove that it is differentiable. Therefore each indicatrix body can be translated in such a way that the translated body is balanced and we always have an associated Finsler manifold with balanced indicatrices. Finsler manifolds having balanced indicatrices represent a class of Finsler spaces such that the so-called Brickell's conjecture holds \cite{Brickell}, see also \cite{Vincze3}.
\end{abstract}
\maketitle

\section{Introduction}

\subsection{Homogeneous functions} Let $\mathbb{R}^n$ be the standard real coordinate space equipped with the canonical inner product. The canonical coordinates are denoted by $y^1,\ldots,y^n$. The set $B=\{v\in \mathbb{R}^n\ | \ (v^1)^2+\ldots+(v^n)^2\leq 1\}$
is the unit ball with boundary $\partial B$ formed by Euclidean unit vectors.  The function
$\varphi\colon \mathbb{R}^n\to \mathbb{R}$ is called \emph{positively homogene\-ous of degree $k$} if $\varphi(tv)=t^k\varphi(v)$ for all $v\in \mathbb{R}^n$ and $t>0$.  Let the function $\varphi$ be differentiable away from the origin. Euler's theorem states that it is positively homogeneous of degree $k$ if and only if
$C\varphi =k\varphi$, where
$$C:=y^1\frac{\partial}{\partial y^1}+\ldots+y^n\frac{\partial}{\partial y^n}$$
is the so-called Liouville vector field. Note that the partial derivatives of $\varphi$ are positively homogeneous functions of degree $k-1$. The analytic properties are reduced at the origin in general. It is known that if a function of class $C^k$ is positively homogeneous of degree $k\geq 0$ then it must be polynomial of degree $k$. In what follows we suppose that homogeneous functions are at least continuously differentiable away from the origin unless otherwise stated.

\subsection{Minkowski functionals and associated objects} Let $K\subset \mathbb{R}^n$ be a convex body containing the origin in its interior and consider a non-zero element $v\in \mathbb{R}^n$.
The value $L(v)$ of the \emph{Minkowski functional} induced by $K$ is defined as the only positive real number such that $v/L(v)\in \partial K$, where $\partial K$ is the boundary of $K$; for the general theory of convex sets and Minkowski spaces see \cite{Lay} and \cite{Thompson}.

\begin{Def}
The function $L$ is called a \emph{Finsler-Minkowski functional} if each non-zero element $v\in \mathbb{R}^n$ has an open neighbourhood such that the restricted function is of class at least $\ \mathcal{C}^4$ and the Hessian matrix
$$g_{ij}=\frac{\partial^2 E}{\partial y^j \partial y^i}$$
of the energy function $E:=(1/2)L^2$ is positive definite.
\end{Def}

\noindent
I. \cite{BCS} \cite{Matsu1} Since the differentiation decreases the degree of homoge\-ni\-ty in a systematical way we have that $g_{ij}$'s are positively homogeneous of degree zero. To express the infinitesimal change of the inner product it is usual to introduce the (lowered) Cartan tensor by the components
\begin{equation}
\mathcal{C}_{ijk}=\frac{1}{2}\frac{\partial g_{ij}}{\partial y^k}=\frac{1}{2}\frac{\partial ^3 E}{\partial y^k \partial y^j \partial y^i}.
\end{equation}
It is totally symmetric and $y^k\mathcal{C}_{ijk}=0$
because of the zero homogenity of $g_{ij}$. Using its inverse $g^{ij}$ we can introduce the quantities $\mathcal{C}^k_{ij}=g^{kl}\mathcal{C}_{ijl}$ to express the covariant derivative
$$\nabla_X Y=\left(X^i\frac{\partial Y^k}{\partial y^i}+X^iY^j\mathcal{C}^k_{ij}\right)\frac{\partial}{\partial y^k}$$
with respect to the Riemannian metric $g$. The formula is based on the standard Lévi-Civita process and repeated pairs of indices are automatically summed as usual. The curvature tensor $\mathbb{Q}$ can be written as
$$Q_{ijk}^l=\mathcal{C}^l_{sk}\mathcal{C}^s_{ij}-\mathcal{C}^l_{sj}\mathcal{C}^s_{ik}.$$ 
Since
$$y^ig_{ij}=\frac{\partial E}{\partial y^j}=L\frac{\partial L}{\partial y^j},\ \ \textrm{i.e.}\ \ g(C,Y)=YE=L(YL),$$
we have that the Liouville vector field is an outward-pointing unit normal to the indicatrix hypersurface $\partial K:=L^{-1}(1)$. On the other hand
$\nabla_X C=X$ which means that the indicatrix is a totally umbilical hypersurface and $\textrm{div}\  C=n$. 

\vspace{0.2cm}
\noindent
II. Consider the volume form
$$d\mu=\sqrt{\det g_{ij}}\ dy^1\wedge \ldots \wedge dy^n$$
and let $f$ be a homogeneous function of degree zero. We have by Euler's theorem that
$$\textrm{div}\ (fC)=f\textrm{div}\ C+Cf=nf.$$
Using the divergence theorem (with the Liouville vector field as an outward-pointing unit normal to the indicatrix hypersurface) for the vector field $X=fC$ it follows that
\begin{equation}
\label{sec}
\int_K f \, d\mu=\frac{1}{n}\int_{\partial K}f\, \mu,
\end{equation}
where
$$\mu=\iota_C d\mu=\sqrt{\det g_{ij}}\ \sum_{i=1}^n (-1)^{i-1} y^i dy^1\wedge\ldots\wedge dy^{i-1}\wedge dy^{i+1}\ldots \wedge dy^n$$
denotes the induced volume form on the indicatrix hypersurface\footnote{According to the isolated singularity at the origin, the integral over the indicatrix body can be accurately taken as the limit
$$\int_K f \, d\mu:=\lim_{U \to \{{\bf{0}}\}} \int_{K\setminus U} f \, d\mu $$
as $U$ schrinks to the origin. Since the continuity (away from the origin) implies that $f$ attains both its minima and maxima at the points of $\partial U$
and zero homogenity (constant values along the rays emanating from the origin) provides them to be global minima and maxima, the limit obviously exists; recall that the volume element $\sqrt{\det g_{ij}}$ is also homogeneous of degree zero.}. We have the following method to calculate integrals of the form $\int_{\partial K} f\mu$, where the integrand $f$ is a zero homogeneous function \cite{Vincze3}. Consider the mapping
$$v \mapsto T(v):=\varphi(v)v,\ \ \textrm{where}\ \ \varphi(v)=\frac{|v|}{L(v)},$$ 
as a diffeomorphism (away from the origin); $|v|$ denotes (for example) the usual Euclidean norm of vectors in $\mathbb{R}^n$. Then
$$\int_{ K} f\, d \mu=\int_{T^{-1}(K)}\varphi^nf\circ \varphi \sqrt{\det g_{ij}}\circ \varphi=\int_{B}\varphi^nf\, d\mu$$
because of the zero homogenity of the integrand. Therefore
\begin{equation}
\int_{B}\varphi^nf\, d\mu=\int_{ K} f\, d \mu
\end{equation}
and, by equation (2),
\begin{equation}
\int_{\partial B}\varphi^nf\, \mu=\int_{\partial K} f\, \mu.
\end{equation} 
We can introduce the following averaged inner products
\begin{equation}
\label{averaged1}
\gamma_1(v,w):=\int_{\partial K}g(v,w)\, \mu \ \ \textrm{and}\ \ \gamma_2(v,w)=\int_{\partial K} m(v,w)\, \mu
\end{equation}
on the vector space, where $m(v,w)=g(v,w)-(V L)(W L)$ is the angular metric tensor,
$$V=v^1\frac{\partial}{\partial y^1}+\ldots+v^n\frac{\partial}{\partial y^n}\ \ \textrm{and}\ \ W=w^1\frac{\partial}{\partial y^1}+\ldots+w^n\frac{\partial}{\partial y^n}.$$
Furthermore
$$\gamma_3(v,w)=\gamma_1(v,w)-\gamma_2(v,w)=\int_{\partial K} (V L)(W L)\, \mu.$$
The theory of averaging and its applications is a relatively new and popular trend in Finsler-Minkowski geometry with a rapidly increasing number of papers; \cite{Vincze1}, \cite{Vincze2} and \cite{Tor}, see also \cite{Ai}, \cite{C2}, \cite{Mat2}, \cite{Vincze3} and \cite{V7}. As a recent contribution to the topic see \cite{C1} which contains several candidates to be averaged together with an extensive overview about the history of averaging in Finsler geometry. The origin goes back to the alternative proof of Szabó's theorem on the Riemann metrizability of Berwald manifolds \cite{Vincze1} and the solution of Matsumoto's problem on conformally equivalent Berwald manifolds \cite{Vincze2}, see also \cite{V5}. In \cite{Vincze3} 
some new steps were taken by introducing the associated Randers-Minkowski functionals. These are given by a linear perturbation of the associated Riemannian metric. The linear term is defined as
$$\beta(v):=\int_{\partial K} VL\, \mu, \ \ \textrm{where}\ \ V=v^1\frac{\partial}{\partial y^1}+\ldots+v^n\frac{\partial}{\partial y^n};$$
for the details see \cite{Vincze3}. Using the divergence theorem we have that
$$\int_K \textrm{div}\ \big(LV-(VL)C\big)\, d\mu=0$$
because the vector field $LV-(VL)C$ is tangential to the indicatrix hypersurface: $\big(LV-(VL)C\big)L=L(VL)-(VL)L=0.$ Since $C(VL)=0$
$$\int_K \textrm{div}\ (LV)\, d\mu=\int_K \textrm{div}\ ((VL)C)\, d\mu=n\int_K VL \, d\mu \stackrel{(2)}{=}\int_{\partial K} VL\,  \mu.$$
On the other hand
$$\textrm{div}\ (LV)=VL+Lv^iC_i,$$
where $C_i=g^{jk}C_{ijk}$. Formula (\ref{sec}) says that 
$$n\int_{K} \textrm{div}\ (LV)\, d\mu=\int_{\partial K} \textrm{div}\ (LV)\, \mu=\int_{\partial K} VL+Lv^iC_i\, \mu$$
and thus
\begin{equation}
\label{perturb}
v^i\int_{\partial K} LC_i\, \mu=(n-1)\int_{\partial K} VL\, \mu=(n-1)\beta(v).
\end{equation}

\begin{Def} \emph{\cite{Vincze3}}
The body $K$ is called balanced if $\beta=0$.
\end{Def}

Finsler-Minkowski functionals with balanced indicatrices represent a class of spaces such that the so-called Brickell's conjecture holds \cite{Brickell}, see also \cite{Vincze3}. The main theorem in \cite{Vincze3} states that if $L$ is a Finsler-Minkowski functional with a balanced indicatrix body of dimension at least three and the Lévi-Civita connection $\nabla$ has zero curvature then $L$ is a norm coming from an inner product, i.e. the Minkowski vector space reduces to a Euclidean one.
The original version was proved by F. Brickell \cite{Brickell} Theorem 1 (see also \cite{Schneider}) using the stronger condition of absolute homogenity (the symmetry of $K$ with respect to the origin) instead of the balanced indicatrix body. \emph{It seems possible that the equation \emph{(1)} imply that the functions $Y^h$ are homogeneous linear functions. If this were so, Theorem 1 would follow from \emph{(3)} under the weaker condition of positive homogenity} \cite{Brickell}, p. 327. The proof of the generalized theorem is essentially based on Santaló's inequality and its applications in Finsler-Minkowski geometry; see e.g. \cite{Duran}. General investigations on the volume of the unit spheres in a Finsler space can be found in \cite{BS}. In what follows we summarize the theoretical background and the theory will be used in case of a Funk manifold. 

\subsection{Finsler spaces} \cite{BCS}, \cite{Matsu1} and \cite{Shen3} \emph{Finsler geometry is a non-Rieman\-ni\-an geometry
in a finite number of dimensions. The differentiable structure is the same as the Riemannian
one but distance is not uniform in all directions. Instead of the Euclidean spheres in the tangent
spaces, the unit vectors form the boundary of general convex sets containing the origin in their interiors.} (M. Berger) 

Let $M$ be a differentiable manifold with local coordinates $u^1$, ..., $u^n$ on $U\subset M$. The induced coordinate system on the tangent manifold consists of the functions 
$$x^1:=u^1\circ \pi, \ldots, x^n:=u^n\circ \pi\ \ \textrm{and}\ \ y^1:=du^1, \ldots,y^n:=du^n,$$ 
where $\pi \colon TM\to M$ is the canonical projection. A Finsler structure on a differentiable manifold $M$ is a smoothly varying family $F\colon TM\to \mathbb{R}$ of Finsler-Minkowski functionals in the tangent spaces satisfying the following conditions:
\begin{itemize}
\item each non-zero element $v\in TM$ has an open neighbourhood such that the restricted function is of class at least $\ \mathcal{C}^4$ in all of its variables $x^1$, ..., $x^n$ and $y^1$, ..., $y^n$,
\item the Hessian matrix of the energy function $E:=(1/2)F^2$ with respect to the variables $y^1, \ldots, y^n$ is positive definite. 
\end{itemize}

\vspace{0.2cm}
\noindent
I. Let $f\colon TM\to \mathbb{R}$ be a zero homogeneous function and let us define the average-valued function \cite{C1}
$$A_f(p):=\int_{\partial K_p} f\,  \mu_p,$$
where $\partial K_p$ is the indicatrix hypersurface belonging to the Finsler-Minkow\-ski functional of the tangent space and $\mu_p$ is the restriction of the volume form
$$\sqrt{\det g_{ij}}\ \sum_{i=1}^n (-1)^{i-1} y^i dy^1\wedge\ldots\wedge dy^{i-1}\wedge dy^{i+1}\ldots \wedge dy^n$$
to the Cartesian product $T_pM \times \ldots \times T_pM$. We use the sript $p$ to express that canonical objects of a Finsler-Minkowski functional are taken point by point\footnote{ Note that integrals of the form
$$\int_{K}f\, d\mu_p=\int_{y(K)}f\circ y^{-1}\sqrt{\det g_{ij}}\circ y^{-1}\, dy^1\ldots dy^n$$
are independent of the choice of the coordinate system (orientation). Actually, the orientation is convenient but not necessary to make integrals of functions sense \cite{Warner}.}.

\vspace{0.2cm}
\noindent
II. {\bf{Horizontal distributions}} \cite{Grif1}, \cite{Grif2} and \cite{Szil2}. To compute the partial derivatives of average-valued functions we need the notion of horizontal distributions: using compatible collections  $G_i^k$ of functions on local neighbourhoods of the tangent manifold let us define the (horizontal) vector fields
\begin{equation}
X_i^h=\frac{\partial}{\ \partial x^i}-G_i^k\frac{\partial}{\ \partial y^k}\ \ (i=1, \ldots, n).
\end{equation}
\begin{Def} 
The horizontal distribution $h$ is a collection of subspaces spanned by the vectors $X_i^h$ as the base point runs through the non-zero elements of the tangent manifold. If the functions $G_i^k$ are positively homogeneous of degree $1$ then the distribution is called \emph{homogeneous}. In case of
\begin{equation}
\label{grad}
\frac{\partial G_i^k}{\partial y^j}=\frac{\partial G_j^k}{\partial y^i}
\end{equation}
we say that $h$ is \emph{torsion-free}\footnote{According to the basic results of the classical vector calculus, condition (\ref{grad}) says that $G^k$'s are the coordinates of a gradient-type vector field with respect to the variables $y^1, \ldots, y^n$.}. The horizontal distribution is \emph{conservative} if the derivatives of $F$ vanish into the horizontal directions.
\end{Def}
According to Section 3 in \cite{Vincze3}
$$\frac{\partial A_f}{\partial u^i}_p=\int_{\partial K_p}-nf X_i^h ln F+X_i^h f+
f\frac{1}{2}g^{mn}X_i^hg_{mn}+f\frac{\partial G_i^k}{\partial y^k}\, \mu_p.$$
In terms of index-free expressions 
\begin{equation}
X_pA_f=\int_{\partial K_p}-nf X^h ln F+X^h f+f\tilde{\mathcal{C}}^{'}(X^c)\, \mu_p,
\end{equation}
where $\tilde{\mathcal{C}}^{'}$ is the semibasic trace of the second Cartan tensor
$$g(\mathcal{C}'(X_i^c,X_j^c),X_k^v)=$$
$$=\frac{1}{2}\left(X_i^h g(X_j^v,X_k^v)-g([ X_i^h,X_j^v],X_k^v)-g(X_j^v,[ X_i^h,X_k^v])\right)$$
associated to $h$. Recall that $X^v$, $X^c$ and $X^h$ are the vertical, complete and horizontal lifts of the vector field $X$ on the base manifold. Especially 
$$X_i^v:=\frac{\partial }{\ \partial y^i},\ \ X_i^c:=\frac{\partial }{\ \partial x^i}\ \ \textrm{and}\ \ X_i^h=\frac{\partial}{\ \partial x^i}-G_i^k\frac{\partial}{\ \partial y^k}.$$

\begin{Cor} \cite{Vincze3}
If the horizontal distribution is conservative then we have
the reduced formula
\begin{equation}
\label{basicformula}
X_pA_f=\int_{\partial K_p}X^h f+f\tilde{\mathcal{C}^{'}}(X^c)\, \mu_p.
\end{equation}
\end{Cor}

In what follows we shall use the canonical horizontal distribution of the Finsler manifold which is uniquely determined by the following conditions: it is conservative, torsion-free and homogeneous. 

\vspace{0.2cm}
\noindent
III. Consider the associated Riemannian metric tensors
$$\gamma_1(X_p,Y_p)=\int_{\partial K_p} g(X^v,Y^v)\, \mu_p,\ \ \gamma_2(X_p,Y_p)=\int_{\partial K_p} m(X^v,Y^v)\, \mu_p,$$
where 
$$m(X^v,Y^v)=g(X^v,Y^v)-(X^vL)(Y^vL)$$
is the angular metric tensor and
$$\gamma_3(X_p,Y_p)=\gamma_1(X_p,Y_p)-\gamma_2(X_p,Y_p)=\int_{\partial K_p} (X^vF)(Y^vF)\, \mu_p.$$
The weighted versions are
$$\Gamma_i(X_p,Y_p)=\frac{1}{\textrm{Area}\ (\partial K_p)}\gamma_i(X_p,Y_p),$$
where $i=1, 2$, $3$ and  
$$\textrm{Area}\ (\partial K_p)=\int_{\partial K_p} 1\, \mu_p.$$
After introducing the $1$-form
$$\beta(X_p)=\int_{\partial K_p} X^vF\, \mu_p$$
we have the Randers-Minkowski functionals
$$F_1(v):=\sqrt{\Gamma_1(v,v)}+\frac{\beta(v)}{\textrm{Area}\ (\partial K_{\pi(v)})} \ \ \textrm{and}\ \ F_3(v):=\sqrt{\Gamma_3(v,v)}+\frac{\beta(v)}{\textrm{Area}\ (\partial K_{\pi(v)})}$$
associated to the Finsler space \cite{Vincze3}.

\subsection{Funk metrics} \cite{O}, \cite{Shen2} and \cite{Shen3} Let $K\subset \mathbb{R}^n$ be a convex body containing the origin in its interior and suppose that the induced function $L$ is a Finsler-Minkowski functional on the vector space $\mathbb{R}^n$. By changing the origin in the interior of $K$ we have a smoothly varying family of Finsler-Minkowski functionals parameterized by the interior points of $K$:
\begin{equation}
L\left(p+\frac{v}{L_p(v)}\right)=1,
\end{equation}
where the script refers to the base point of the tangent vector $v$. Let $U$ be the interior of the indicatrix body. The manifold $U$ equipped with the Finslerian fundamental function
$$F\colon TU=U\times \mathbb{R}^n\to \mathbb{R},\ v_p\mapsto F(v_p):=L_p(v)$$
is called a Funk space (or Funk manifold). It is a special Finsler manifold.
Another notations and terminology: $K_p$ denotes the unit ball with respect to the functional $L_p$ and $\partial K_p$ is its boundary. Especially $K_{{\bf 0}}=K$ and $\partial K_{{\bf 0}}=\partial K$,
$$F(v_{\bf{0}}):=L_{\bf{0}}(v)=L(v)\ \ \textrm{and}\ \ \frac{\partial L}{\partial u^i}_{v}=\frac{\partial F}{\partial y^i}_{v_{\bf{0}}},$$
where $u^1$, $\ldots$, $u^n$ denotes the canonical coordinate system of $\mathbb{R}^n$ restricted to $U$ as a base manifold. Recall that the Riemann-Finsler metric
$$g_{ij}=\frac{1}{2}\frac{\partial^2 F^2}{\partial y^i \partial y^j}$$
provides any tangent space (away from its origin) to be a Riemannian manifold; $\partial K_p$ equipped with the usual induced Riemannian structure is a totally umbilical Riemannian submanifold of the corresponding tangent space (see subsection 1.2).
\begin{Thm} \emph{\cite{Vincze3}}
For any $p\in U$ the indicatrix hypersurfaces $\partial K_p$ and $\partial K$ are conformal to each other as Riemannian submanifolds in the corresponding tangent spaces. The conform mapping between these structures comes from the projection
$$\rho(v_p):=p+\frac{v}{L_p(v)}\in \partial K$$
of the tangent space $TU$. Especially
$$g_{\rho(v_p)}(w,z)=\left(1-p^k\frac{\partial L}{\partial u^k}_{\rho(v_p)}\right)g_{v_p}(w,z),$$
where $w$ and $z$ are tangential to the indicatrix hypersurface $\partial K$ at $\rho(v_p)$, i.e they are tangential to $\partial K_p$ at $v_p$ too.
\end{Thm}
\begin{center}
\begin{figure}
\includegraphics[viewport=0 0 595 842, scale=0.15]{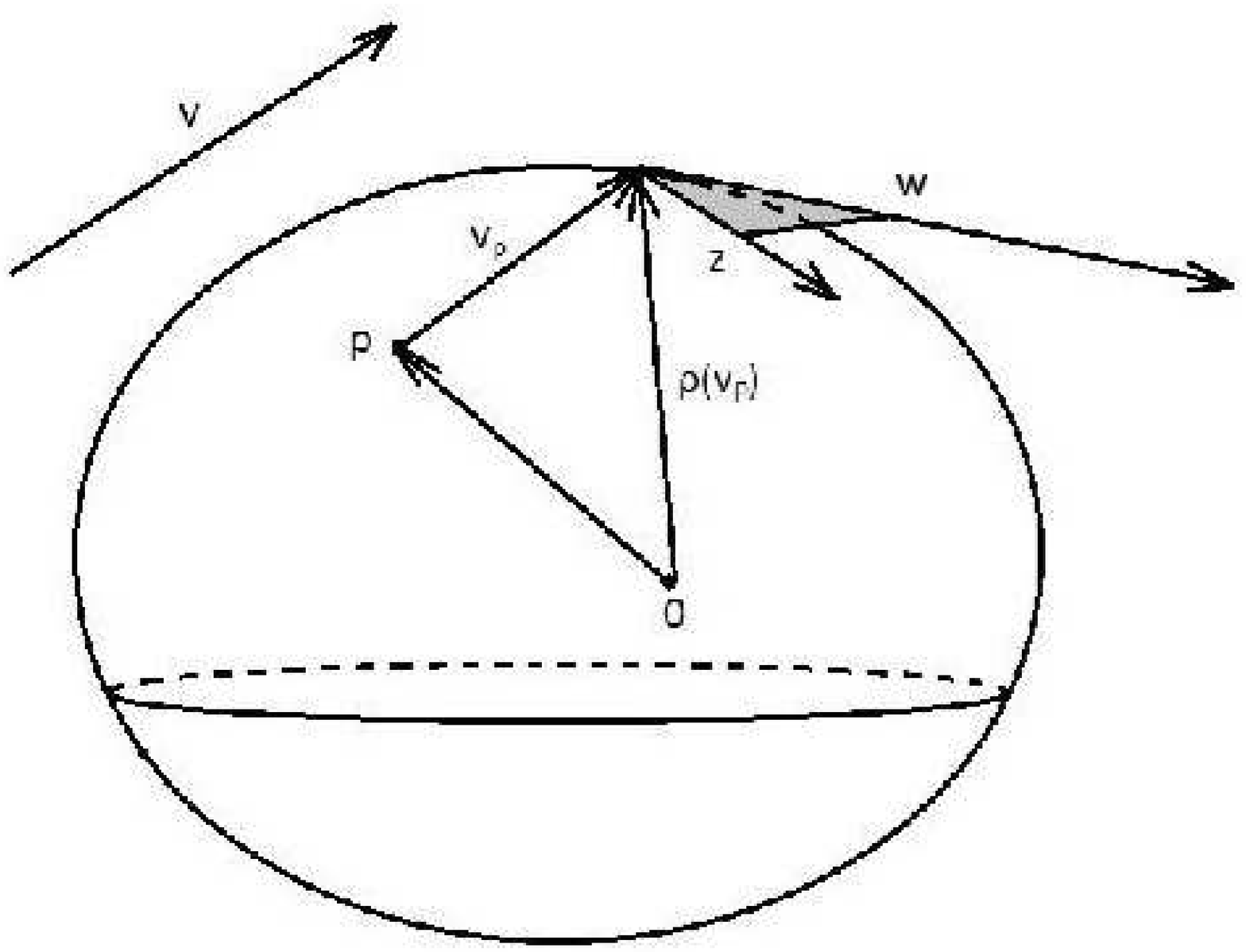}
\caption{Conformality}
\end{figure}
\end{center}
Further important relationships between the canonical data of a Funk space are based on Okada's theorem \cite{O}:
\begin{equation}
\label{okada}
0=\frac{\partial F}{\partial y^i}-\frac{1}{F}\frac{\partial F}{\partial x^i} \ \Rightarrow\ F\frac{\partial F}{\partial y^i}=\frac{\partial F}{\partial x^i}.
\end{equation}
Okada's theorem is a rule how to change derivatives with respect to $x^i$ and $y^i$. This results in relatively simple formulas for the canonical objects of the Funk manifold. In what follows we are going to summarize some of them (proofs are straightforward calculations \cite{Shen3}):
\begin{equation}
G^k=\frac{1}{2}y^kF,\ \ G_i^k=\frac{\partial G^k}{\partial y^k}=\frac{F}{2}\delta_{i}^k+\frac{1}{2}y^k \frac{\partial F}{\partial y^i},
\end{equation}
\begin{equation}
\label{canhordistr}
X_i^h=\frac{\partial}{\ \partial x^i}-G_i^k\frac{\partial}{\ \partial y^k}=\frac{\partial}{\ \partial x^i}-\left(\frac{F}{2}\delta_{i}^k+\frac{1}{2}y^k \frac{\partial F}{\partial y^i}\right)\frac{\partial}{\ \partial y^k}
\end{equation}
for the canonical horizontal distribution. The first and the second Cartan tensors are related as  
\begin{equation}
\label{cartantensors}
\mathcal{C}'=\frac{1}{2}F\mathcal{C}
\end{equation}
and the curvature of the canonical horizontal distribution can be expressed in the following form 
\begin{equation}
R\left(\frac{\partial}{\ \partial x^i},\frac{\partial}{\ \partial x^i}\right)=\frac{F}{4}\left(\frac{\partial F}{\ \partial y^i}\frac{\partial}{\ \partial y^j}-\frac{\partial F}{\ \partial y^j}\frac{\partial}{\ \partial y^i}\right).
\end{equation}
In terms of lifted vector fields
\begin{equation}
R(X^c,Y^c)=\frac{1}{4}\bigg{(}g(X^v,C)Y^v-g(Y^v,C)X^v\bigg{)}.
\end{equation}
\noindent
Using relation (\ref{cartantensors}) we can specialize the basic formula (\ref{basicformula}) for derivatives of average-valued functions:
\begin{equation}
\label{specialder}
X_pA_f=\int_{\partial K_p} X^h f+\frac{F}{2}f\tilde{\mathcal{C}}(X^c)\, \mu_p.
\end{equation} 

\section{Analytic properties of the area function of a Funk metric}

In what follows we apply formula (\ref{specialder}) in the special case of the area function
\begin{equation}
r\colon U\to \mathbb{R}, \ \ p\mapsto r(p):=A_1(p)=\int_{\partial K_p} 1\, \mu_p.
\end{equation}

At first we are going to investigate the partial derivatives at the points of the interior of $K$. 

\begin{Thm}
For any $p\in \textrm{int}\ K$ 
\begin{equation}
\label{partial}
\frac{\partial^{m} r}{\partial u^{1 i_1}\ldots \partial u^{n i_n}}_p=c_m\int_{{\partial}_p K}\left(\frac{\partial F}{\partial y^1}\right)^{i_1}\cdot \ldots \cdot \left(\frac{\partial F}{\partial y^n} \right)^{i_n}\, \mu_p,
\end{equation}
where 
$$c_0:=1,\ \ c_m:=\frac{(n-1)(n+1)\cdot \ldots \cdot ((n-1)+2(m-1))}{2^m},$$
$0\leq i_1, \ldots, 0\leq i_n$ and $m=i_1+\ldots+i_n$. 
\end{Thm}

The proof is based on the induction with respect to $m$. In case of $m=1$ formula (\ref{specialder}) shows that
\begin{equation}
\frac{\partial r}{\partial u^i}_p =\frac{1}{2}\int_{\partial K_p} F\tilde{\mathcal{C}}\left(\frac{\partial}{\partial x^i}\right)\, \mu_p=\frac{n-1}{2}\int_{\partial K_p} \frac{\partial F}{\partial y^i}\, \mu_p
\end{equation}
because of subsection 1.2/II, formula (\ref{perturb}). In other words the formula implies that
$$d r=\frac{n-1}{2}\beta$$
in case of a Funk manifold \cite{Vincze3}. The second order partial derivatives of the area function coincide the associated Riemannian metric $\gamma_3$ up to a constant proportional term:
\begin{equation}
\label{second}
\frac{\partial^2 r}{\partial u^j \partial u^i}_p=\frac{n^2-1}{4}\int_{\partial K_p} \frac{\partial F}{\partial y^i}\frac{\partial F}{\partial y^j}\,\mu_p;
\end{equation}
for the details see \cite{Vincze3}. 

\begin{Cor} The area function is strictly convex.
\end{Cor}

\begin{Cor} The body $K$ is balanced if and only if the area function of the associated Funk space has a global minimizer at the origin.
\end{Cor}
Suppose that (\ref{partial}) is true up to the order $m$. Differentiating again:
{\small{$$\frac{\partial}{\partial u^k}_p\left(\frac{\partial^{m} r}{\partial u^{1 i_1}\ldots \partial u^{n i_n}}\right)\stackrel{(\ref{specialder})}{=}$$
$$c_m\sum_{j=1}^n i_j \int_{{\partial}_p K}\left(\frac{\partial F}{\partial y^1}\right)^{i_1}\cdot \ldots \cdot \left(\frac{\partial F}{\partial y^j}\right)^{i_j-1}\cdot \ldots \cdot \left(\frac{\partial F}{\partial y^n}\right)^{i_n}\cdot \left(\frac{\partial}{\partial u^k}\right)^h \frac{\partial F}{\partial y^{j}}\, \mu_p+$$
$$\frac{c_m}{2}\int_{{\partial}_p K}F \left(\frac{\partial F}{\partial y^1}\right)^{i_1}\cdot \ldots \cdot \left(\frac{\partial F}{\partial y^1}\right)^{i_n} \tilde{\mathcal{C}}\left(\frac{\partial}{\partial x^k}\right)\, \mu_p.$$}}

\noindent
Differentiating Okada's relation (\ref{okada}) it follows that
$$\frac{\partial^2 F}{\partial y^j \partial x^k}=\frac{\partial F}{\partial y^k}\frac{\partial F}{\partial y^j}+F\frac{\partial^2 F}{\partial y^k \partial y^j}$$
and thus (by formula (\ref{canhordistr}) and zero homogenity)
$$\left(\frac{\partial}{\partial u^k}\right)^h \frac{\partial F}{\partial y^j}=\frac{\partial^2 F}{\partial x^k \partial y^j}-\frac{1}{2}F\frac{\partial^2 F}{\partial y^k \partial y^j}-\frac{1}{2}\frac{\partial F}{\partial y^k}y^r\frac{\partial^2 F}{\partial y^r \partial y^j}=$$
$$\frac{\partial F}{\partial y^k}\frac{\partial F}{\partial y^j}+\frac{1}{2}F\frac{\partial^2 F}{\partial y^k \partial y^j}.$$
We have that 
{\small{$$\frac{\partial}{\partial u^k}_p\left(\frac{\partial^{m} r}{\partial u^{1 i_1}\ldots \partial u^{n i_n}}\right)=$$
$$c_m \sum_{j=1}^n i_j\int_{{\partial}_p K}\left(\frac{\partial F}{\partial y^1}\right)^{i_1}\cdot \ldots \cdot \left(\frac{\partial F}{\partial y^j}\right)^{i_j}\cdot \ldots \cdot \left(\frac{\partial F}{\partial y^j}\right)^{i_k+1}\cdot \ldots \cdot \left(\frac{\partial F}{\partial y^n}\right)^{i_n}\, \mu_p+$$
$$\frac{c_m}{2}\sum_{j=1}^n i_j \int_{{\partial}_p K}F\left(\frac{\partial F}{\partial y^1}\right)^{i_1}\cdot \ldots \left(\frac{\partial F}{\partial y^j}\right)^{i_j-1}\cdot \ldots \cdot \left(\frac{\partial F}{\partial y^n}\right)^{i_n}\cdot \frac{\partial^2 F}{\partial y^k \partial y^j}+$$
$$\frac{c_m}{2}\int_{{\partial}_p K}F \left(\frac{\partial F}{\partial y^1}\right)^{i_1}\cdot \ldots \cdot \left(\frac{\partial F}{\partial y^n}\right)^{i_n}  
\tilde{\mathcal{C}}\left(\frac{\partial}{\partial x^k}\right)\, \mu_p.$$}}

\noindent
To substitute the last terms in the integrand. Let us define the vector field
{\small{$$T_k:=F\left(\frac{\partial F}{\partial y^1}\right)^{i_1}\cdot \ldots \cdot \left(\frac{\partial F}{\partial y^n}\right)^{i_n}\frac{\partial}{\partial y^k}-
\left(\frac{\partial F}{\partial y^n}\right)^{i_1}\cdot \ldots \cdot \left(\frac{\partial F}{\partial y^k}\right)^{i_k+1}\cdot \ldots \cdot \left(\frac{\partial F}{\partial y^1}\right)^{i_n}C.$$}}

\noindent
Since $T_kF=0$ the divergence theorem says that
\begin{equation}
\label{div}
\int_{\partial_p K}\ \textrm{div}\ T_k\, \mu_p=0,
\end{equation}
where
$$\textrm{div}\ T_k=F\left(\frac{\partial F}{\partial y^1}\right)^{i_1}\cdot \ldots \cdot \left(\frac{\partial F}{\partial y^n}\right)^{i_n}  
\tilde{\mathcal{C}}\left(\frac{\partial}{\partial x^k}\right)+$$
$$\left(\frac{\partial F}{\partial y^1}\right)^{i_1}\cdot \ldots \cdot \left(\frac{\partial F}{\partial y^k}\right)^{i_k+1}\cdot \ldots \cdot \left(\frac{\partial F}{\partial y^n}\right)^{i_n}+$$
$$F \sum_{j=1}^n i_j\left(\frac{\partial F}{\partial y^1}\right)^{i_1}\cdot \ldots \cdot \left(\frac{\partial F}{\partial y^j}\right)^{i_j-1}\cdot \ldots \cdot \left(\frac{\partial F}{\partial y^n}\right)^{i_n}\cdot \frac{\partial^2 F}{\partial y^k \partial y^j}-$$
$$n\left(\frac{\partial F}{\partial y^1}\right)^{i_1}\cdot \ldots \cdot \left(\frac{\partial F}{\partial y^k}\right)^{i_k+1}\cdot \ldots \cdot \left(\frac{\partial F}{\partial y^n}\right)^{i_n}$$
because of the zero homogenity:
$$C\left( \left(\frac{\partial F}{\partial y^1}\right)^{i_1}\cdot \ldots \cdot \left(\frac{\partial F}{\partial y^k}\right)^{i_k+1}\cdot \ldots \cdot \left(\frac{\partial F}{\partial y^1}\right)^{i_n}\right)=0.$$
Integrating both sides, formula (\ref{div}) gives that 
$$(n-1)\int_{\partial_p K}\left(\frac{\partial F}{\partial y^1}\right)^{i_1}\cdot \ldots \cdot \left(\frac{\partial F}{\partial y^k}\right)^{i_k+1}\cdot \ldots \cdot \left(\frac{\partial F}{\partial y^n}\right)^{i_n}\, \mu_p=$$
$$\sum_{j=1}^n i_j \int_{\partial_p K} F\left(\frac{\partial F}{\partial y^1}\right)^{i_1}\cdot \ldots \cdot \left(\frac{\partial F}{\partial y^j}\right)^{i_j-1}\cdot \ldots \cdot \left(\frac{\partial F}{\partial y^n}\right)^{i_n}\cdot \frac{\partial^2 F}{\partial y^k \partial y^j}\, \mu_p+$$
$$\int_{\partial_p K} F\left(\frac{\partial F}{\partial y^1}\right)^{i_1}\cdot \ldots \cdot \left(\frac{\partial F}{\partial y^n}\right)^{i_n} 
\tilde{\mathcal{C}}\left(\frac{\partial}{\partial x^k}\right)\, \mu_p.$$
Therefore
$$\frac{\partial}{\partial u^k}_p\left(\frac{\partial^{m} r}{\partial u^{1 i_1}\ldots \partial u^{n i_n}}\right)=$$
$$c_m\left(\sum_{j=1}^n i_j+\frac{n-1}{2}\right)\int_{\partial_p K}\left(\frac{\partial F}{\partial y^1}\right)^{i_1}\cdot \ldots \cdot \left(\frac{\partial F}{\partial y^k}\right)^{i_k+1}\cdot \ldots \cdot \left(\frac{\partial F}{\partial y^n}\right)^{i_n}\, \mu_p=$$
$$c_m\frac{2m+n-1}{2}\int_{\partial_p K}\left(\frac{\partial F}{\partial y^1}\right)^{i_1}\cdot \ldots \cdot \left(\frac{\partial F}{\partial y^k}\right)^{i_k+1}\cdot \ldots \cdot \left(\frac{\partial F}{\partial y^n}\right)^{i_n}\, \mu_p=$$
$$c_m\frac{(n-1)+2(m+1-1)}{2}\int_{\partial_p K}\left(\frac{\partial F}{\partial y^1}\right)^{i_1}\cdot \ldots \left(\frac{\partial F}{\partial y^k}\right)^{i_k+1}\cdot \ldots \cdot \left(\frac{\partial F}{\partial y^n}\right)^{i_n}\, \mu_p=$$
$$c_{m+1}\int_{\partial_p K}\left(\frac{\partial F}{\partial y^1}\right)^{i_1}\cdot \ldots \left(\frac{\partial F}{\partial y^k}\right)^{i_k+1}\cdot \ldots \cdot \left(\frac{\partial F}{\partial y^n}\right)^{i_n}\, \mu_p$$
and the induction is completed. 

To finish this section recall that the projection $\rho$ provides a nice connection between the indicatrix hypersurfaces $\partial K_p$ and $\partial K$. Using Theorem 1 we have
\begin{equation}
\label{volint}
r(p)=\int_{\partial K_p} 1\, \mu_p=\int_{\partial K} \left(1-p^k\frac{\partial L}{\partial u^k}\right)^{-\frac{n-1}{2}}\, \mu,
\end{equation}
where $\mu=\mu_{\bf{0}}$ is the canonical volume form associated to $\partial K$. From the general theory of Minkowski functionals \cite{BCS}
\begin{equation}
\label{geomseries}
w^k\frac{\partial L}{\partial u^k}_v \leq L(w)\ \Rightarrow \ -L(-p) \leq  p^k\frac{\partial L}{\partial u^k}_v \leq L(p) < 1
\end{equation}
for any nonzero element $v$. 

\begin{Thm} The function $r$ is analytic at the origin in the sense that
\begin{equation}
\label{analytic}
r(p)=r({\bf{0}})+\sum_{m=1}^{\infty} \ \sum_{i_1+ \ldots+ i_n=m} \frac{1}{i_1!\cdot \ldots \cdot i_n!} \frac{\partial^{m} r}{\partial u^{1 i_1}\ldots \partial u^{n i_n}}_{\bf{0}}p_1^{i_1}\cdot \ldots \cdot p_n^{i_n}
\end{equation}
for any point $p$ in the interior of $K \cap (- K)$. Especially if $K$ is symmetric about the origin then the area function is analytic in the interior of $K$.
\end{Thm}

\begin{proof} Consider the function
$$f_0(x)=\left(\frac{1}{1-x}\right)^{\frac{n-1}{2}};$$
since
$$f_1(x):=f_0'(x)=\frac{n-1}{2}\frac{f_0(x)}{1-x}$$
a simple induction shows that
$$f_m(x):=f_0^m(x)=c_m\frac{f_0(x)}{\ (1-x)^m},$$
where
$$c_m=\frac{(n-1)(n+1)\cdot \ldots \cdot ((n-1)+2(m-1))}{2^m}.$$
According to the convergence radius $|x|<1$ we have that
$$f(x)=1+\sum_{m=1}^{\infty}\frac{c_m}{\ m!}x^m.$$
From the polynomial theorem 
$$\left(p^k\frac{\partial L}{\partial u^k}\right)_v^m=\sum_{i_1+ \ldots+ i_n=m} \frac{m!}{i_1!\cdot \ldots \cdot i_n!} \left(\frac{\partial L}{\partial u^1} \right)^{i_1}_v\cdot \ldots \cdot \left(\frac{\partial L}{\partial u^n}\right)^{i_n}_v p_1^{i_1}\cdot \ldots \cdot p_n^{i_n}$$
and, consequently, for any $v\in \partial K$ we have that
$$\left(1-p^k\frac{\partial L}{\partial u^k}\right)_v^{-\frac{n-1}{2}}=$$
$$1+\sum_{m=1}^{\infty} c_m \sum_{i_1+ \ldots+ i_n=m}\frac{1}{i_1!\cdot \ldots \cdot i_n!} \left(\frac{\partial L}{\partial u^1} \right)^{i_1}_v\cdot \ldots \cdot \left(\frac{\partial L}{\partial u^n}\right)^{i_n}_v p_1^{i_1}\cdot \ldots \cdot p_n^{i_n}$$
provided that $p$ is close enough to the origin. Especially if $p$ is in the interior of $K \cap (- K)$ then (\ref{geomseries}) implies that 
$$-1 < p^k\frac{\partial L}{\partial u^k}_v < 1\ \ \ \  (v\in \partial K)$$
and the convergence is uniform. Therefore
we can integrate the series member by member:
{\small{$$r(p)\stackrel{(\ref{volint})}{=}\int_{\partial K} \left(1-p^k\frac{\partial L}{\partial u^k}\right)^{-\frac{n-1}{2}}\, \mu=r({\bf{0}})+$$
$$\sum_{m=1}^{\infty} c_m \sum_{i_1+ \ldots+ i_n=m}\frac{1}{i_1!\cdot \ldots \cdot i_n!} \int_{\partial K}\left(\frac{\partial L}{\partial u^1}\right)^{i_1}\cdot \ldots \cdot \left(\frac{\partial L}{\partial u^n}\right)^{i_n}\, \mu \ p_1^{i_1}\cdot \ldots \cdot p_n^{i_n}=$$
$$r({\bf{0}})+\sum_{m=1}^{\infty} \sum_{i_1+ \ldots+ i_n=m}\frac{1}{i_1!\cdot \ldots \cdot i_n!}  \frac{\partial^{m} r}{\partial u^{1 i_1}\ldots \partial u^{n i_n}}_{\bf{0}} p_1^{i_1}\cdot \ldots \cdot p_n^{i_n}$$}}

\noindent
because of formula (\ref{partial}) for $p={\bf{0}}$.
\end{proof}

\begin{Rem} \emph{Since we can choose any interior point $p_0$ of $K$ as the origin, a similar formula holds in the interior of the intersection of $K_{p_0}$ and $-K_{p_{0}}$, where the $-$ operator means the reflection about the "origin" $p_{0}$. Therefore}
{\small{
$$r(p)=r(p_{0})+$$
$$\sum_{m=1}^{\infty} \ \sum_{i_1+ \ldots+ i_n=m} \frac{1}{i_1!\cdot \ldots \cdot i_n!} \frac{\partial^{m} r}{\partial u^{1 i_1}\ldots \partial u^{n i_n}}_{p_{0}}(p_1-p_{01})^{i_1}\cdot \ldots \cdot (p_n-p_{0n})^{i_n}.$$}}
\end{Rem}

\section{The asymptotic behavior of the area function of a Funk metric}

In what follows we prove that the area can be arbitrarily large near to the boundary of $K$.

\begin{Thm} $$\lim_{L(p)\to 1^-}r(p)=\infty, $$
i.e. for any real number $M$ there is an $\varepsilon > 0$ such that $L(p)>1-\varepsilon$ implies that $r(p)>M$. 
\end{Thm}

Let $B$ be the unit ball with respect to the Euclidean inner product
$$\langle v,w \rangle=g_{p/L(p)}(v,w),\ \ \ |v|:=\sqrt{g_{p/L(p)}(v,w)}.$$
Under these notations
\begin{equation}
\label{volint1}
r(p)\stackrel{(24)}{=}\int_{\partial K} \left(1-p^k\frac{\partial L}{\partial u^k}\right)^{-\frac{n-1}{2}}\, \mu=\int_{\partial B} \varphi^n \left(1-p^k\frac{\partial L}{\partial u^k}\right)^{-\frac{n-1}{2}}\, \mu,
\end{equation}
where $\varphi(v)=|v|/L(v)$, see formula (4). Since $\varphi$ is zero homogeneous (i.e. it is constant along the rays emanating from the origin) the global minimum $k_p$ is attained at a point of $\partial K$. By continuity properties 
$$\frac{\sqrt{g_{q/L(q)}(v,v)}}{L(v)}> \frac{k_p}{2}$$
for any $q\in U_{p/L(p)}\cap \partial K$, where $U_{p/L(p)}$ is an open neighbourhood of the Finslerian unit vector $p/L(p)$; see the shaded region around $p/L(p)$ in Figure 2. By compactness we have a finite covering of $\partial K$ by the sets  $U_{p_1/L(p_1)}, \ldots, U_{p_m/L(p_m)}$ and, consequently,
$$k_1:=\min \{k_{p_1}, \ldots, k_{p_m}\}$$
is a positive constant such that $\varphi(v)>k_1$ independently of the choice  $p\neq {\bf{0}}$. Therefore
$$r(p) > k_1^n\int_{\partial B} \left(1-p^k\frac{\partial L}{\partial u^k}\right)^{-\frac{n-1}{2}}\, \mu.$$
Since
$$p^k\frac{\partial L}{\partial u^k}=\frac{\partial L}{\partial {\bf{n}}},$$
where $\partial/\partial {\bf{n}}$ means the directional derivative along the unit vector $e_n=p/L(p)$,
we are motivated to use an $g_{p/L(p)}$ - orthonormal basis $e_1, \ldots, e_{n-1},e_n$ for differentiation, i.e.
$$u^1(p)= \ldots =u^{n-1}(p)=0,\ \ u^n(p)=L(p)\ \ \textrm{and}\ \ \frac{\partial L}{\partial u^n}=\frac{\partial L}{\partial {\bf{n}}}.$$
The parameterization of $\partial B$ can be given as
$$H\times \left [ -\frac{\pi}{2}, \frac{\pi}{2}\right ] \mapsto \mathbb{R}^n,\ \ (u, t)\mapsto (\sigma(u)\cos t, \sin t),$$
where $\sigma \colon H\to \mathbb{R}^{n-1}$ is the parameterization of the intersection of $\partial B$ with the hyperplane $u^n=0$. Therefore
{\small{$$\int_{\partial B} \left(1-p^k\frac{\partial L}{\partial u^k}\right)^{-\frac{n-1}{2}}\, \mu=$$
$$\int_{H} u\mapsto \left(\int_{-\frac{\pi}{2}}^{\ \frac{\pi}{2}}
\left(1-L(p)\frac{\partial L}{\partial u^n}\right)_{(\sigma(u)\cos t, \sin t)}^{-\frac{n-1}{2}}\cos^{n-2}t\cdot \sqrt{\det g_{ij}}_{(\sigma(u)\cos t, \sin t)}\, dt\right)\, d\sigma,$$}}
where
$$g_{ij}=g\left(\frac{\partial}{\partial u^i}, \frac{\partial}{\partial u^j}\right).$$
Since $\sqrt{\det g_{ij}}$ is homogeneous of degree zero we have a positive constant $k_2$ such that
$$\sqrt{\det g_{ij}}_v \geq k_2\ \ \ (v\in \partial K).$$
In a similar way as above, continuity and compactness provides us to choose $k_2$ independently of $p\neq {\bf{0}}$. Therefore
$$r(p) > k_1^n\cdot k_2 \int_{H} u\mapsto \ \left(\int_{-\frac{\pi}{2}}^{\ \frac{\pi}{2}}
\left(1-L(p)\frac{\partial L}{\partial u^n}\right)_{(\sigma(u)\cos t, \sin t)}^{-\frac{n-1}{2}}\cos^{n-2}t\, dt\right)\, d\sigma.$$
In what follows we are going to investigate the integral
$$\int_{-\frac{\pi}{2}}^{\ \frac{\pi}{2}}
\left(1-L(p)\frac{\partial L}{\partial u^n}\right)_{(\sigma(u)\cos t, \sin t)}^{-\frac{n-1}{2}}\cos^{n-2}t\, dt.$$

\begin{Rem} \emph{The investigations include the case of $n=2$. The set $H$ reduces to $\{-1,1\}$ and the integration with respect to $\sigma$ means a summation under the possible values
$\sigma=\pm 1$, i.e.}
{\small{$$\int_{H} u\mapsto \left(\int_{-\frac{\pi}{2}}^{\ \frac{\pi}{2}}
\left(1-L(p)\frac{\partial L}{\partial u^2}\right)_{(\sigma(u)\cos t, \sin t)}^{-\frac{1}{2}}\, dt\right)\, d\sigma=$$}}
{\small{$$\int_{-\frac{\pi}{2}}^{\ \frac{\pi}{2}}
\left(1-L(p)\frac{\partial L}{\partial u^2}\right)_{(\cos t, \sin t)}^{-\frac{1}{2}}\, dt+ \int_{-\frac{\pi}{2}}^{\ \frac{\pi}{2}}
\left(1-L(p)\frac{\partial L}{\partial u^2}\right)_{(-\cos t, \sin t)}^{-\frac{1}{2}}\, dt.$$}}
\end{Rem}

\begin{Lem} For any parameter $u$ the function 
$$w_u(t):=\frac{\partial L}{\partial u^n}_{(\sigma(u)\cos t, \sin t)}$$
is strictly monotone increasing in the variable $t$ and tends to $1$ as $t\to \frac{\pi}{2}^{-}$.
\end{Lem}

\begin{proof}
Note that $w_u(t)$ is the last coordinate of the Euclidean gradient which is an outward-pointing unit normal to the indicatrix hypersurface with respect to the Euclidean inner product $g_{p/L(p)}$. By differentiation
$$w_u'(t)=-\sigma^{i}(u)\sin(t)\frac{\partial^2 L}{\partial u^i \partial u^n}_{(\sigma(u)\cos(t), \sin(t))}+\cos(t)\frac{\partial^2 L}{\partial u^n \partial u^n}_{(\sigma(u)\cos(t), \sin(t))}.$$
Using the zero homogenity of the partial derivatives of $L$ it follows that
$$\sigma^{j}(u)\cos(t)\frac{\partial^2 L}{\partial u^j \partial u^i}_{(\sigma(u)\cos(t), \sin(t))}+\sin(t)\frac{\partial^2 L}{\partial u^n \partial u^i}_{(\sigma(u)\cos(t), \sin(t))}=0$$
and, consequently,
\begin{equation}
\label{Lhospital}
w_u'(t)=
\end{equation}
$$\cos(t)\left(\sigma^{i}(u)\sigma^j(u)\frac{\partial^2 L}{\partial u^i \partial u^j}_{(\sigma(u)\cos(t), \sin(t))}+\frac{\partial^2 L}{\partial u^n \partial u^n}_{(\sigma(u)\cos(t),\sin(t))}\right)>0$$
away from $t=\pi/2$ or $-\pi/2$. Moreover (by the homogenity of degree one)
$$\frac{\partial L}{\partial u^n}_{(0,1)}=L(0,1)=L\left(p/L(p)\right)=1$$
and the proof is finished.
\end{proof}

\begin{center}
\begin{figure}
\includegraphics[viewport=0 0 595 842, scale=0.2]{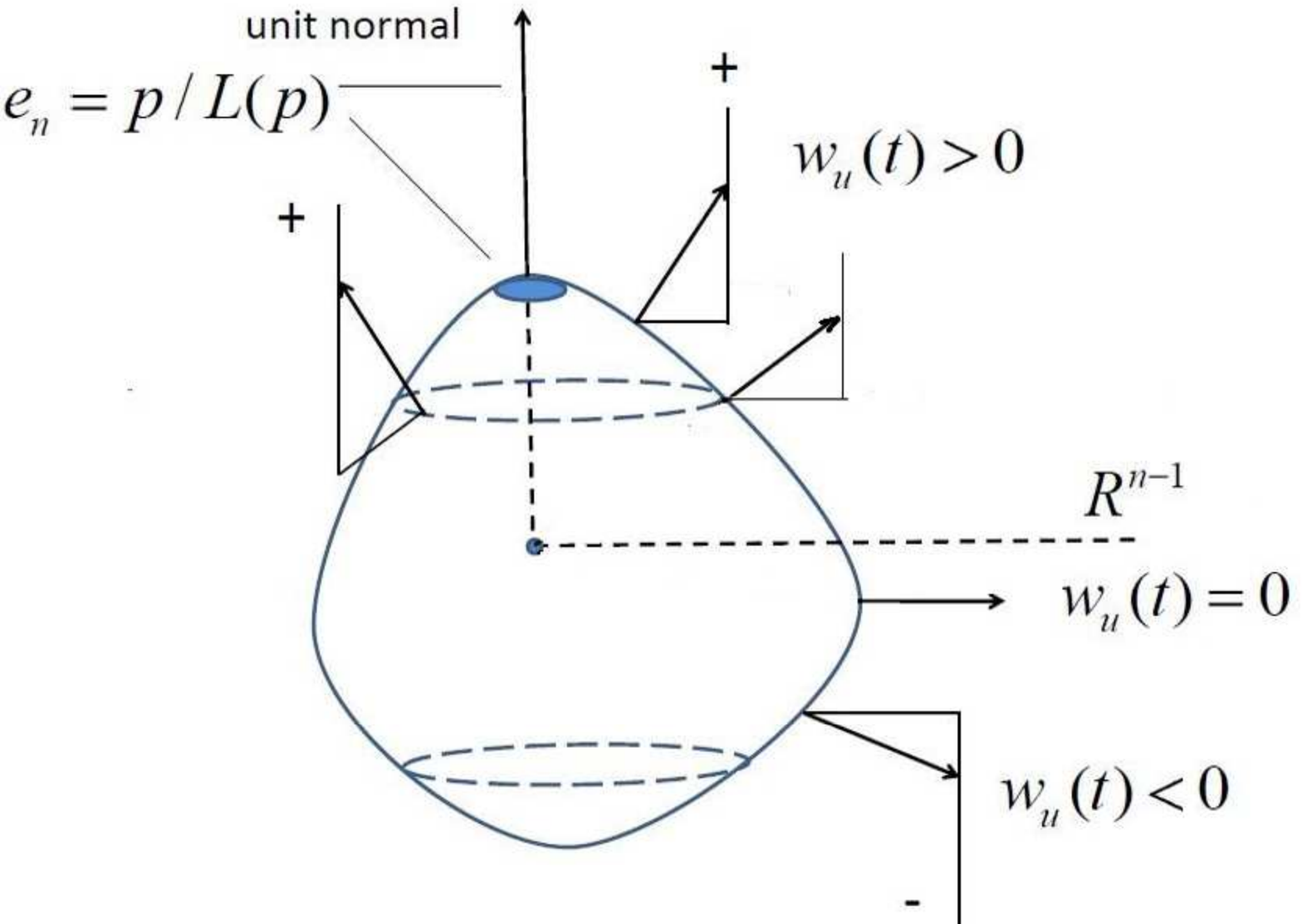}
\caption{Continuity and compactness.}
\end{figure}
\end{center}

Since $w_u(\frac{\pi}{2})=1$ we have that
\begin{equation}
\label{lowerbound}
\sup_{u\in H} \inf \left\{t\in \left[-\frac{\pi}{2}, \frac{\pi}{2}\right] \ | \ w_u(t) >0 \right\} =:t_0 < \frac{\pi}{2},
\end{equation}
where $t_0$ can be choosen independently of $p\neq {\bf{0}}$. The proof needs continuity and compactness\footnote{Note that the Euclidean inner product $g_{p/L(p)}$ is changing as the point $p$ is varying in the interior of $K$ but it is constant into radial directions, i.e. along the integral curves of the Liouville vector field. Therefore  inequality (\ref{lowerbound}) can be stated at first in a local neighbourhood (continuity; see the shaded region around $p/L(p)$ in Figure 2) and it can be extended for any $p\neq {\bf{0}}$ by using a finite covering of $\partial K$.}. Geometrically if we turn around by the parameter $u$ then the Euclidean gradient is pointed into the upper half space as Figure 2 shows. On the other hand if $L(p)$ is sufficiently close to $1$ then 
\begin{equation}
\label{upperbound}
t_0 < \inf_{u\in H} \sup \left\{t\in \left[-\frac{\pi}{2}, \frac{\pi}{2}\right] \ | \ w_u(t) < L(p)\right\} =:t_1 < \frac{\pi}{2}.
\end{equation}
For any $t\in [t_0, t_1]$
\begin{equation}
\label{estimate1}
w_u^2(t)=\left(\frac{\partial L}{\partial u^n}\right)^2_{(\sigma(u)\cos t, \sin t)} < L(p) \frac{\partial L}{\partial u^n}_{(\sigma(u)\cos t, \sin t)}
\end{equation}
which means that
$$r(p) > k_1^n\cdot k_2 \int_{H} u\mapsto \ \left(\int_{-\frac{\pi}{2}}^{\ \frac{\pi}{2}}
\left(1-L(p)\frac{\partial L}{\partial u^n}\right)_{(\sigma(u)\cos t, \sin t)}^{-\frac{n-1}{2}}\cos^{n-2}t\, dt\right)\, d\sigma >$$
$$k_1^n\cdot k_2 \int_{H} u\mapsto \ \left(\int_{t_0}^{t_1}
\left(\frac{\cos^2(t)}{1-w_u^2(t)}\right)^{\frac{n-1}{2}}\frac{1}{\cos (t)}\, dt\right)\, d\sigma$$

\begin{Lem} For any parameter $u$ the function
$$v_u(t):=\frac{1-w_u^2(t)}{\cos^2(t)}$$
tends to $1$ as $t\to \frac{\pi}{2}^-$, i.e.
$$\lim_{t\to \frac{\pi}{2}^-}\frac{1-w_u^2(t)}{\cos^2(t)}=1.$$
\end{Lem}

\begin{proof}
According to Lemma 1 we can use L'Hospital rule as follows
$$\lim_{t\to \frac{\pi}{2}^-}v_u(t)=\lim_{t\to \frac{\pi}{2}^-}\frac{-2w_u(t) w_u'(t)}{-2\cos t \sin t}\stackrel{(\ref{Lhospital})}{=}$$
$$\frac{\partial L}{\partial u^n}_{(0,1)}\left(\sigma^{i}(u)\sigma^j(u)\frac{\partial^2 L}{\partial u^i \partial u^j}_{(0,1)}+\frac{\partial^2 L}{\partial u^n \partial u^n}_{(0,1)}\right),$$
where 
$$\frac{\partial L}{\partial u^n}_{(0,1)}=L(0,1)=1\ \ \textrm{and}\ \ \frac{\partial^2 L}{\partial u^n \partial u^n}_{(0,1)}=0$$
because of homogenity properties. Since for any parameter $u$ 
$${\bf{z}}_u=\sigma^{i}(u)\frac{\partial}{\partial u^i}_{(0,1)}$$
is a unit vector with respect to $g_{p/L(p)}$ and ${\bf{z}}_u$  is tangential to $\partial B$ at ($0,1$) it follows that
$$\lim_{t\to \frac{\pi}{2}^-}v_u(t)=m_{p/L(p)}({\bf{z}}_u, {\bf{z}}_u)=1,$$
where $m$ is the angular metric tensor.
\end{proof}

By Lemma's 1 and 2 
$$0<k_3:=\sup_{H\times [t_0,\frac{\pi}{2}]}v_u(t) < \infty$$
and the constant $k_3$ can be choosen independently of $p\neq {\bf{0}}$ (continuity and compactness). We have that 
$$r(p)>k_1^n\cdot k_2 \int_{H} u\mapsto \ \left(\int_{t_0}^{t_1}
\left(\frac{\cos^2(t)}{1-w_u^2(t)}\right)^{\frac{n-1}{2}}\frac{1}{\cos (t)}\, dt\right)\, d\sigma >$$
$$k_1^n\cdot k_2\cdot k_3^{-\frac{n-1}{2}} \int_{H} u\mapsto \ \left(\int_{t_0}^{t_1}
\frac{1}{\cos (t)}\, dt\right)\, d\sigma=\omega_{n-2}\cdot k_1^n\cdot k_2\cdot k_3^{-\frac{n-1}{2}}\int_{t_0}^{t_1}
\frac{1}{\cos (t)}\, dt$$
where $\omega_{n-2}$ is the Euclidean area of the ($n-2$)-dimensional Euclidean unit sphere. Finally 
$$\int_{t_0}^{t_1}
\frac{1}{\cos (t)}\, dt=\left[\ln \tan\left(\frac{t}{2}+\frac{\pi}{4}\right)\right]_{t_0}^{t_1}=\ln \tan\left(\frac{t_1}{2}+\frac{\pi}{4}\right)-\ln \tan\left(\frac{t_0}{2}+\frac{\pi}{4}\right)$$
and thus
$$\lim_{L(p)\to 1^-} r(p)\geq $$
$$\omega_{n-2}\cdot k_1^n\cdot k_2\cdot k_3^{-\frac{n-1}{2}} \cdot\left(\lim_{t_1\to \frac{\pi}{2}^-} \ln \tan \left(\frac{t_1}{2}+\frac{\pi}{4}\right)-\ln \tan\left(\frac{t_0}{2}+\frac{\pi}{4}\right)\right)=\infty$$
as was to be proved. 

\section{Applications to Finsler manifolds}

In what follows we consider a Finsler manifold $M$ equipped with the Finslerian fundamental function $F$. At each point of the manifold we can take the interior of the indicatrix body $K_p$ as a Funk manifold. Let us define the function
$$R\colon v_p\in \textrm{int\ } K_p \to R(v_p):=r(v_p),$$
where $r$ is the area function of the Funk manifold induced by the indicatrix body $K_p$ in the tangent space $T_pM$. Using a local neighbourhood
we can also introduce the mapping
$$Y(q^1, \ldots, q^n, v^1, \ldots, v^n):=\left(\frac{\partial R}{\partial y^1}_{v_q}, \ldots, \frac{\partial R}{\partial y^n}_{v_q}\right),$$
where
$u(q)=(q^1, \ldots, q^n)$ and $v_q=v^i\frac{\partial}{\partial u^i}_{q}$.
By Corollary 2
$$\det \left(\frac{\partial Y^i}{\partial y^j}\right)=\det \left(\frac{\partial^2 r}{\partial y^j\partial y^j}\right)\neq 0$$
and the implicit function theorem allow us to conclude that there exists a smooth function $g\colon \mathcal{U}\to \mathbb{R}^n$ such that $Y(q^1, \ldots, q^n, g(q^1, \ldots, q^n))=0$. Therefore the minimizers of the area functions in the tangent spaces can be expressed as
$$V_q=g^1(q^1, \ldots, q^n)\frac{\partial}{\partial u^1}_{q}+\ldots+g^n(q^1, \ldots, q^n)\frac{\partial}{\partial u^n}_{q}$$
because of the vanishing of the partial derivatives of $R$ with respect to $y^1$, ..., $y^n$. 

\begin{Thm}
Let $M$ be a Finsler manifold with Finslerian fundamental function $F$ and consider the interiors of the indicatrix bodies as Funk manifolds in the tangent spaces. The mapping which sends any point to the uniquely determined minimizer of the area function of the corresponding Funk manifold constitutes a smooth vector field $V$ on the base manifold. The manifold $M$ equipped with the Finslerian fundamental function $F_{V}$ defined by
$$F\left(V_p+\frac{v_p}{F_V(v_p)}\right)=1\ \ \ (p\in M)$$
has balanced indicatrices.  
\end{Thm}

\begin{Rem}
\emph{Finsler manifolds having balanced indicatrices represent a class of Finsler spaces such that the so-called Brickell's conjecture holds; see \cite{Brickell} and \cite{Vincze3}.}
\end{Rem}

\section{Example I - Randers manifolds}

\noindent
I. The family of the unit balls of a Randers manifold is given by translations of Riemannian unit balls. Analytically the Minkowski functionals are coming from a Riemannian metric tensor by using one-form (linear) perturbation in the tangent spaces. This important type of Finsler manifolds was introduced by G. Randers in 1941. Randers manifolds occour naturally in physical applications related to electron optics, navigation problems \cite{BRS} or the Lagrangian of relativistic electrons \cite{AIM}. According to the importance of these applications Randers manifolds are a prosperous subject of the investigations up to this day - see e.g. \cite{CS}.

\vspace{0.2cm}
\noindent
II. Consider the Randers manifold
$$F(v_p):=\sqrt{\alpha_p(v_p,v_p)}+\beta(v_p),$$
where $\alpha$ is a Riemannian metric and $\beta$ is a 1-form satisfying the condition
$$\sup_{\alpha_p(v_p,v_p)=1} \beta(v_p) <1$$
for any $p\in M$. Using the dual vector field $\beta^{\sharp}$ defined by $\alpha(\beta^{\sharp}, Y)=\beta(Y)$
we have that
\begin{equation}
\label{example1}
V=-\frac{\beta^{\sharp}}{1-\|\beta^{\sharp}\|^2}.
\end{equation}
To prove the formula for the minimizing vector field consider an $\alpha_p$ - orthonormal basis $e_1, \ldots, e_n$ in the tangent space $T_pM$ such that
$\beta^{\sharp}_p=\beta_n e_n$. The equation of the indicatrix hypersurface $\partial K_p$ is
$$(y^1)^2+\ldots+(y^{n-1})^2+(1-\beta_n^2)\left(y^n+\frac{\beta_n}{1-\beta_n^2}\right)^2=\frac{1}{1-\beta_n^2},$$ 
where $\beta_n$ is the (only) surviving component of $\beta^{\sharp}_p$. It is a quadric \cite{Shen2} centered at the point
$$y^1=\ldots=y^{n-1}=0\ \ \textrm{and}\ \ y^n=-\frac{\beta_n}{1-\beta_n^2}$$
which is just the value of $V$ at $p$. Therefore $F_V$ restricted to $T_pM$ is a Riemannian fundamental function and the partial derivatives 
$$\frac{\partial F_V}{\partial y^1}, \ldots, \frac{\partial F_V}{\partial y^n}$$
have vanishing integrals on a centered Euclidean sphere. The area $\omega_{n-1}$ of the ($n-1$) - dimensional Euclidean unit sphere is the minimum of the area function. 

\section{Example 2}

\noindent
I. Consider the standard Euclidean ball $B$ with respect to the canonical inner product of $\mathbb{R}^n$. It is known \cite{Shen2} that the Funk space induced by $B$ has a Randers functional
$$F(v_p)=L_p(v)=\frac{1}{1-|p|^2}\sqrt{|v|^2(1-|p|^2)+\langle v, p \rangle^2}+\frac{\langle v, p\rangle}{1-|p|^2},$$
where $p$ is an interior point of $B$,
$$|p|^2=\langle p, p \rangle,\ \ \beta_p(v):=\frac{\langle v, p\rangle}{1-|p|^2},$$
$$\alpha_p(v,w)=\frac{1}{1-|p|^2}\langle v, w\rangle+\frac{\langle v, p \rangle}{1-|p|^2} \cdot \frac{\langle w, p\rangle}{1-|p|^2}.$$
As usual we omit the identification of the base point of the tangent vectors of the manifold $\mathbb{R}^n$. In order to use formula (\ref{example1}) we have to compute $\beta^{\sharp}$ and its norm, where both the sharp - operator and the norm are taken with respect to $\alpha$ (which is different from the canonical inner product). We have that
$$\alpha_p(\beta_p^{\sharp},w)=\beta_p(w)=\frac{\langle w, p\rangle}{1-|p|^2}.$$
On the other hand
$$\alpha_p(\beta_p^{\sharp},w)=\frac{1}{1-|p|^2}\langle \beta_p^{\sharp}, w\rangle+\frac{\langle \beta_p^{\sharp}, p \rangle}{1-|p|^2} \cdot \frac{\langle w, p\rangle}{1-|p|^2}$$
which means that
$$\langle p-\beta_p^{\sharp}, w \rangle=\frac{\langle \beta_p^{\sharp}, p\rangle}{1-|p|^2} \left \langle  p, w \right \rangle\ \ \Rightarrow \ \ p-\beta_p^{\sharp}=\frac{\langle \beta_p^{\sharp}, p\rangle}{1-|p|^2}\cdot p.$$
Taking the (canonical) inner product with $p$ it follows that
$$\frac{\langle \beta_p^{\sharp}, p\rangle}{1-|p|^2}=|p|^2\ \ \Rightarrow \ \ \frac{\beta_p^{\sharp}}{1-|p|^2}=p.$$
Finally
$$\|\beta_p^{\sharp}\|=|p|^2\ \ \ \textrm{and}\ \ \ V_p=-\frac{\beta_p^{\sharp}}{1-\| \beta_p^{\sharp}\|^2}=-p$$
which is just the opposite vector field to the Liouville vector field (see Figure 3).

\begin{center}
\begin{figure}
\includegraphics[viewport=0 0 595 842, scale=0.1]{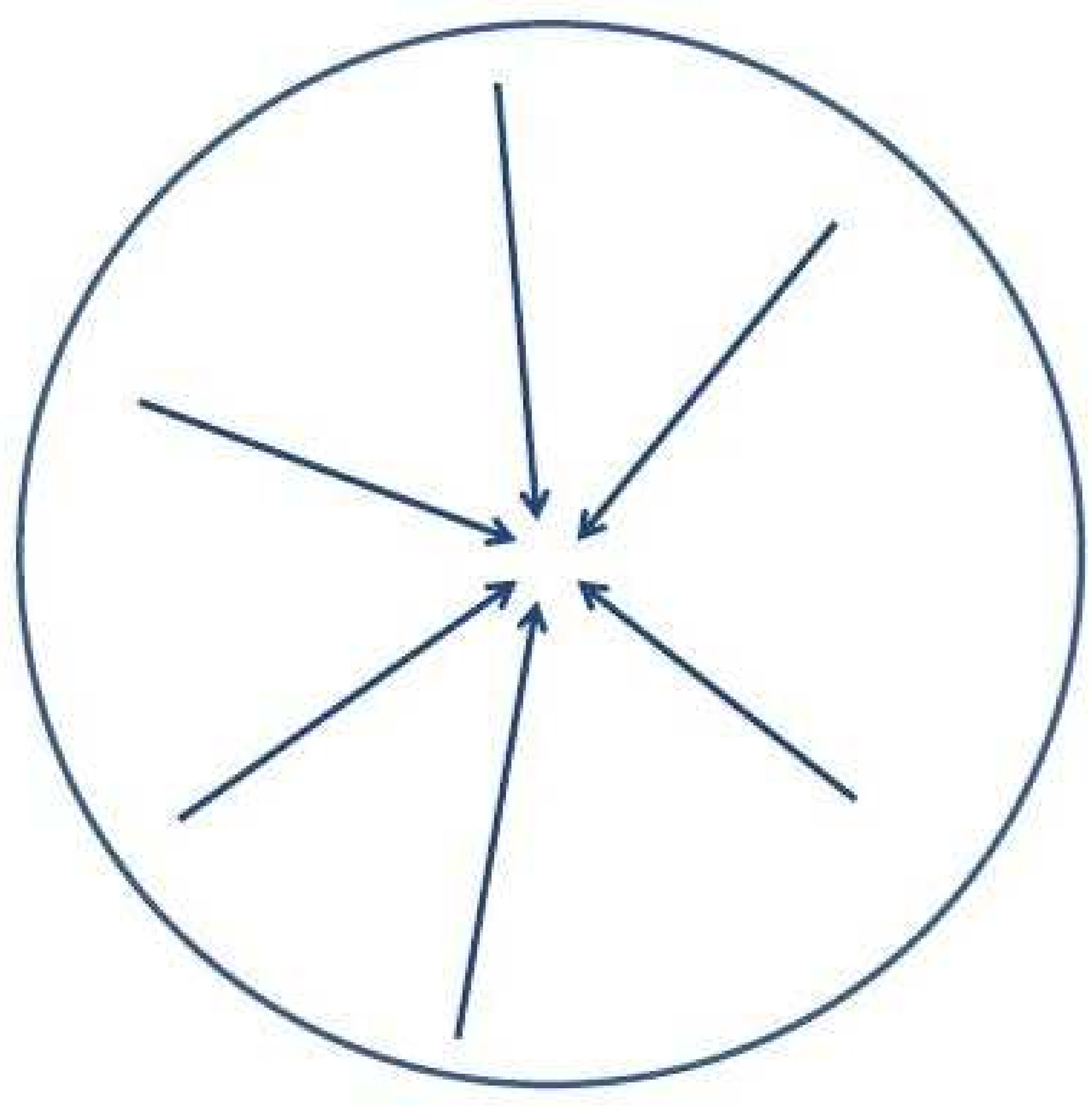}
\caption{The minimizing vector field.}
\end{figure}
\end{center}

\vspace{0.2cm}
\noindent
II. Using formula (\ref{volint}) we can write that for any $p\in \ \textrm{int}\ B$ 
$$r(p)=\int_{\partial B} \left(1-\frac{\langle p, v \rangle}{|v|} \right)^{-\frac{n-1}{2}}\, \mu.$$
If $u^1(p)=\ldots=u^{n-1}(p)=0$ and $u^n(p)=s$ then
$$r(s)=\omega_{n-2}\int_{-\frac{\pi}{2}}^{\ \frac{\pi}{2}} \left(\frac{1}{1-s\cdot \sin (t)}\right)^{\frac{n-1}{2}}\cos^{n-2}(t)\, dt$$
in case of dimension $n\geq 3$. Especially if $n=3$ then 
$$r(s)=\frac{1}{s}\ln \frac{1+s}{1-s},\ \ \textrm{i.e.}\ \ r(p)=\frac{1}{L(p)}\ln \frac{1+L(p)}{1-L(p)}.$$
In general if the dimension is of the form $n=2k+1$ then we can use the standard substitution $x=\tan \frac{t}{2}$ to originate the problem in the integration of partial fractions:
$$t=2\ \textrm{arctan} \ x,\ \ dt=\frac{2}{1+t^2}dx,\ \ \cos(t)=\frac{1-t^2}{1+t^2}\ \ \textrm{and}\ \ \sin(t)=\frac{2t}{1+t^2}.$$

\vspace{0.2cm}
\noindent
III. If $n=2$ then we have that
$$r(p)=\int_{\ 0}^{2\pi} \frac{1}{\sqrt{1-L(p)\cdot \sin (t)}}\, dt$$
which can be numerically integrated by using binomial expansion for negative and fractional powers.

\section*{Acknowledgement}
The work is supported by the University of Debrecen's internal research project RH/885/2013.

\end{document}